\documentclass[12pt]{article}
\usepackage{amsmath,amsthm,amssymb}
\usepackage{xcolor}
\usepackage{hyperref}
\usepackage{ytableau}
\newcommand{\graycell}[1]{*(black!6) {#1}}
\makeatletter

\newif\iffirstvrule@YT

\def\vrule@left@YT{%
  \iffirstvrule@YT
    \global\firstvrule@YTfalse
    \vrule@normal@YT
  \else
    \vrule@none@YT
  \fi
}

\newcommand{\leftnone@YT}[1][]{%
  \def\thisboxcolor@YT{clear}%
  \let\hrule@YT=\hrule@none@YT
  \global\firstvrule@YTtrue
  \let\vrule@YT=\vrule@left@YT
  \startbox@@YT#1\endbox@YT
  \nullfont
}
\def\leftnone@YTbrace#1{\leftnone@YT[#1]}
\def\leftnone{\omit\@ifnextchar[{\leftnone@YT}{\@ifnextchar\bgroup{\leftnone@YTbrace}{\leftnone@YT}}}

\newcommand{\emptybox@YT}[1][]{%
  \def\thisboxcolor@YT{clear}%
  \let\hrule@YT=\hrule@none@YT
  \let\vrule@YT=\vrule@none@YT
  \startbox@@YT#1\endbox@YT
  \nullfont
}
\def\emptybox@YTbrace#1{\emptybox@YT[#1]}
\def\emptybox{\omit\@ifnextchar[{\emptybox@YT}{\@ifnextchar\bgroup{\emptybox@YTbrace}{\emptybox@YT}}}
\makeatother

\hypersetup{
    colorlinks=true,
    linkcolor=purple,
    citecolor=magenta,
    filecolor=cyan,
    urlcolor=cyan,
    pdftitle={Degree-one Pieri rule for nonsymmetric Jack polynomials}
}
\usepackage{aliascnt}
\newtheorem{theorem}{Theorem}[section]
\newaliascnt{lemma}{theorem}
\newtheorem{lemma}[lemma]{Lemma}
\aliascntresetthe{lemma}
\newaliascnt{corollary}{theorem}

\aliascntresetthe{corollary}
\newaliascnt{proposition}{theorem}
\newtheorem{proposition}[proposition]{Proposition}
\aliascntresetthe{proposition}
\newaliascnt{example}{theorem}
\newtheorem{example}[example]{Example}
\aliascntresetthe{example}
\usepackage{cleveref}
\newcommand{\highlight}[1]{\emph{#1}}
\begin{document}

\title{%
Degree-one Pieri rule for nonsymmetric Jack polynomials
}

\author{Waldeck Sch\"{u}tzer}
\date{}
 
\maketitle
\begin{abstract}
We give a combinatorial degree-one Pieri rule for integral nonsymmetric Jack
polynomials.  Its coefficients are products obtained from a \textit{jeu de
fl\`eches} on the source and target diagrams.  The proof uses reflection and
raising recursions, a Cherednik-operator commutator, and hook cancellation.


\end{abstract}


\section{Introduction}

Throughout, $\mathbb N=\{0,1,2,\ldots\}$.

Let $F_\eta$ be the integral nonsymmetric Jack polynomial indexed by the
composition $\eta=(\eta_1,\ldots,\eta_n)$ \cite{MR98k:33040}, and let
$F_\eta^*$ be its dual under Sahi's scalar product \cite{MR98g:05154}.  For
the $k$-th standard basis vector $\varepsilon_k$, write
\begin{equation}
 F_{\varepsilon_k}F_\eta
 =\sum_\lambda g_{k\eta}^\lambda F_\lambda^*.
 \label{eqn-intro-pieri-expansion}
\end{equation}

Following Knop \cite{MR98m:05204}, for
$L=\{j_1<\cdots<j_\ell\}\subseteq\{1,\ldots,n\}$, define
$\lambda=c_L(\eta)$\label{def:operator:c_L} by
\[
 \lambda_{j_p}=\eta_{j_{p+1}}\ (p<\ell),\qquad
 \lambda_{j_\ell}=\eta_{j_1}+1,\qquad
 \lambda_i=\eta_i\ (i\notin L).
\]
The subset $L$ is \emph{maximal with respect to $\eta$}\label{def:maximal:wrt:eta} if
\[
\begin{array}{ll}
\text{(M1)}&\eta_i\ne\eta_{j_1}\quad(i<j_1),\\
\text{(M2)}&\eta_i\ne\eta_{j_p}\quad(j_{p-1}<i<j_p,\ 1<p\le\ell),\\
\text{(M3)}&\eta_i\ne\eta_{j_1}+1\quad(j_\ell<i\le n).
\end{array}
\]
Indices in $L$ will be called \highlight{selected}; all other indices are
\highlight{unselected}.

By homogeneity, $g_{k\eta}^\lambda=0$ unless $|\lambda|=|\eta|+1$.
The nonzero coefficients are characterized as follows.

\begin{theorem}\label{thm-non-van-min}
Let $\eta,\lambda\in\mathbb N^n$, $1\le k\le n$, and
$|\lambda|=|\eta|+1$.  The coefficient $g_{k\eta}^\lambda$ is nonzero if and only if:
\begin{enumerate}
\item[(a)] $\lambda=c_S(\eta)$, where
$S=\{i_1<\cdots<i_s\}=\{i:\eta_i\ne\lambda_i\}$;
\item[(b)] either $i_s\ge k$ or $\eta_i=\eta_{i_1}+1$ for some $i\ge k$;
\item[(c)] if $\eta_1=\cdots=\eta_k=\eta_{i_1}$, then $i_1\le k$.
\end{enumerate}
\end{theorem}

For a box $s=(i,j)$ of the diagram of a composition $\gamma$, define its
\highlight{arm} and \highlight{leg} by
\[
 a_\gamma(s)=\gamma_i-j,\qquad
 l_\gamma(s)=\#\{h<i:j\le\gamma_h+1\le\gamma_i\}
 +\#\{h>i:j\le\gamma_h\le\gamma_i\}.
\]
Following Sahi \cite{MR98g:05154}, define its \highlight{upper} and
\highlight{lower hooks} by
$d_\gamma(s)=\alpha(a_\gamma(s)+1)+l_\gamma(s)+1$ and
$d'_\gamma(s)=d_\gamma(s)-1$.

For a maximal $L=\{j_1<\cdots<j_\ell\}$, put $\lambda=c_L(\eta)$.
The \emph{jeu de fl\`eches} uses the diagrams of $\eta$ and $\lambda$, called
the source and target.  In a selected row place upper hooks in the source and
lower hooks in the target; in an unselected row reverse these choices.

Draw an arrow after each selected row, upward in the source and downward in the
target, with the first source arrow pointing northeast and the last target
arrow southwest.  Launch each arrow toward the next selected row.  A hit in a
selected source row replaces the entry by $1$; a hit in a selected target row
replaces it by $-1$.  Other arrows have no effect.  Let $b_{\eta\lambda}$ be
the product of the resulting entries.

If row $k$ is shielded, suppress any replacement in that row and denote the
product by $b_{\eta\lambda}^{(k)}$.  A shield outside $L$ has no effect.
\begin{theorem}
\label{thm-r19-pieri}
Let $L$ be maximal with respect to $\eta$, put $\lambda=c_L(\eta)$, and
let $1\le k\le n$.  Then
\begin{equation}
 g_{k\eta}^\lambda=
 \mathbf{1}_{\{k\in L\}}(\alpha+k)b_{\eta\lambda}^{(k)}
 +\sum_{\substack{i\in L\\i>k}}b_{\eta\lambda}^{(i)}.
 \label{eqn-r19-uniform-coefficient}
\end{equation}
As $L$ ranges over the maximal subsets with respect to $\eta$,
\eqref{eqn-r19-uniform-coefficient} gives all coefficients whose target is
$c_L(\eta)$; coefficients with any other target are zero.
\end{theorem}

\begin{example}
For $\eta=(0,1,3,1,2)$, $\lambda=(1,3,2,1,1)$, and
$L=\{1,2,3,5\}$, \Cref{fig:jeufl:before,fig:jeufl:after:noshield}
show the initial and unshielded games.  The source arrows replace the
upper hooks $\alpha+1$ and $\alpha+3$ in rows $3$ and $5$ by $1$; the
target arrows replace the lower hooks $\alpha+2$ and $2\alpha+2$ in rows
$1$ and $2$ by $-1$.

Reset the game and shield row $1$.  The target entry $\alpha+2$ is no
longer replaced by $-1$, while the other three changes remain.  Therefore
\[
 b_{\eta\lambda}^{(1)}=-(\alpha+2)b_{\eta\lambda},
\]
as shown in \Cref{fig:jeufl:after:shield}.
\end{example}
\begin{figure}[htb]
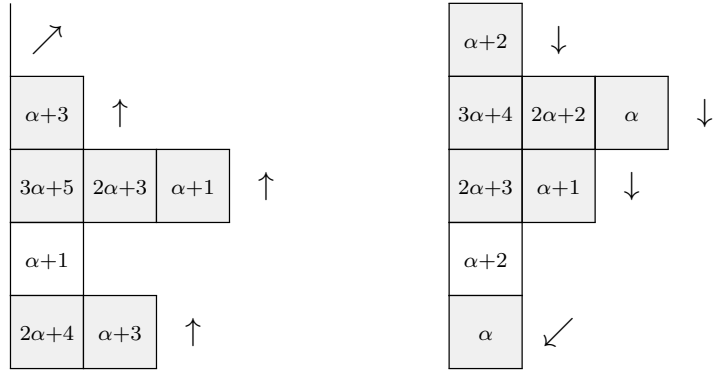

\begin{center}
\ytableausetup{
  boxsize=2.3em,
  nosmalltableaux
}
\begin{ytableau}
\leftnone[\nearrow] \\
\graycell {\scriptstyle \alpha+3} & \emptybox[\uparrow]\\
\graycell {\scriptstyle 3\alpha+5}
  & \graycell {\scriptstyle 2\alpha+3}
  & \graycell {\scriptstyle \alpha+1} & \emptybox[\uparrow]\\
{\scriptstyle \alpha+1}  \\
\graycell {\scriptstyle 2\alpha+4}
  & \graycell {\scriptstyle \alpha+3} & \emptybox[\uparrow]
\end{ytableau} \hspace{4em}
\begin{ytableau}
\graycell {\scriptstyle \alpha+2} & \emptybox[\downarrow] \\
\graycell {\scriptstyle 3\alpha+4} & \graycell {\scriptstyle 2\alpha+2} &
\graycell {\scriptstyle \alpha} & \emptybox[\downarrow]\\
\graycell {\scriptstyle 2\alpha+3}
  & \graycell {\scriptstyle \alpha+1}
  & \emptybox {\downarrow} \\
{\scriptstyle \alpha+2} \\
\graycell {\scriptstyle \alpha}
  & \emptybox[\swarrow]
\end{ytableau}
\end{center}
\caption{The jeu de flèches for $\eta=(0,1,3,1,2)$,
$\lambda=(1,3,2,1,1)$, and $L=\{1,2,3,5\}$, before the exceptional
arrows are launched.  The source is on the left, the target on the right,
and selected rows are shaded.}\label{fig:jeufl:before}
\end{figure}
\begin{figure}[htb]
\begin{center}
\ytableausetup{
  boxsize=2.3em,
  nosmalltableaux
}
\begin{ytableau}
\leftnone[\nearrow] \\
\graycell {\scriptstyle \alpha+3} & \emptybox[\uparrow]\\
\graycell {\scriptstyle 3\alpha+5}
  & \graycell {\scriptstyle 2\alpha+3}
  & \graycell {\scriptstyle 1} & \emptybox[\uparrow]\\
{\scriptstyle \alpha+1} \\
\graycell {\scriptstyle 2\alpha+4}
  & \graycell {\scriptstyle 1} & \emptybox[\uparrow]
\end{ytableau} \hspace{4em}
\begin{ytableau}
\graycell {\scriptstyle -1} & \emptybox[\downarrow] \\
\graycell {\scriptstyle 3\alpha+4} & \graycell {\scriptstyle -1} &
\graycell {\scriptstyle \alpha} & \emptybox[\downarrow]\\
\graycell {\scriptstyle 2\alpha+3}
  & \graycell {\scriptstyle \alpha+1}
  & \emptybox {\downarrow} \\
{\scriptstyle \alpha+2} \\
\graycell {\scriptstyle \alpha}
  & \emptybox[\swarrow]
\end{ytableau}
\[
b_{\eta\lambda}=
{\alpha}^{2} \left( \alpha+1 \right)^{2} \left( \alpha+2
 \right) \left( \alpha+3 \right) \left( 2\,
\alpha+3 \right) ^{2} \left( 2\,\alpha+4 \right) \left( 3\,\alpha+4 \right) \left( 3\,\alpha+5 \right).
\]
\end{center}
\caption{The same jeu de flèches after the arrows are launched.}\label{fig:jeufl:after:noshield}
\end{figure}

\begin{figure}[htb]
\begin{center}
\ytableausetup{
  boxsize=2.3em,
  nosmalltableaux
}
\begin{ytableau}
\leftnone[\nearrow] & \emptybox[\text{Shield}]\\
\graycell {\scriptstyle \alpha+3} & \emptybox[\uparrow]\\
\graycell {\scriptstyle 3\alpha+5}
  & \graycell {\scriptstyle 2\alpha+3}
  & \graycell {\scriptstyle 1} & \emptybox[\uparrow]\\
{\scriptstyle \alpha+1} \\
\graycell {\scriptstyle 2\alpha+4}
  & \graycell {\scriptstyle 1} & \emptybox[\uparrow]
\end{ytableau} \hspace{4em}
\begin{ytableau}
\graycell {\scriptstyle \alpha+2} & \emptybox[\downarrow] & \emptybox[\text{Shield}]\\
\graycell {\scriptstyle 3\alpha+4} & \graycell {\scriptstyle -1} &
\graycell {\scriptstyle \alpha} & \emptybox[\downarrow]\\
\graycell {\scriptstyle 2\alpha+3}
  & \graycell {\scriptstyle \alpha+1}
  & \emptybox {\downarrow} \\
{\scriptstyle \alpha+2} \\
\graycell {\scriptstyle \alpha}
  & \emptybox[\swarrow]
\end{ytableau}
\[
b_{\eta\lambda}^{(1)}=
-{\alpha}^{2} \left( \alpha+1 \right)^{2} \left( \alpha+2
 \right)^2 \left( \alpha+3 \right) \left( 2\,
\alpha+3 \right) ^{2} \left( 2\,\alpha+4 \right) \left( 3\,\alpha+4 \right) \left( 3\,\alpha+5 \right).
\]
\end{center}
\caption{The same jeu de flèches with shield on row 1 after the arrows are launched.}\label{fig:jeufl:after:shield}
\end{figure}

Marshall \cite{MR2004a:33020} gave an explicit formula for multiplication
by one variable, and Forrester--McAnally \cite{qa/0006006} obtained related
Pieri-type formulas.  The combinatorial degree-one rule first appeared in the author's 2004
dissertation \cite{MR2706078}; to the author's knowledge, it was the first
fully combinatorial Pieri formula for nonsymmetric Jack polynomials.  That proof used Marshall's formula and a static filling.  The
jeu de fl\`eches is introduced here, and the proof derives the rule directly
from reflection and raising recursions, the Cherednik commutator, and hook
cancellation.

\section{Preliminaries}
\label{sec-preliminaries}

\subsection{Jack polynomials and duality}

Let $E_\eta=x^\eta+\sum_{\nu<\eta}a_{\nu\eta}x^\nu$ be the monic
nonsymmetric Jack polynomial, with the usual dominance--Bruhat order on
compositions \cite[Section~2]{MR98k:33040}.  Put
\begin{equation}
 \overline\eta_i=\alpha\eta_i-l'_\eta(i),\qquad
 l'_\eta(i)=\#\{h<i:\eta_h\ge\eta_i\}
            +\#\{h>i:\eta_h>\eta_i\}.
 \label{eqn-r19-eigenvalues}
\end{equation}
The Cherednik operators $\xi_1,\ldots,\xi_n$ have $E_\eta$ as a simultaneous
eigenfunction with eigenvalues $\overline\eta_i$
\cite[\S 3 and Corollary~3.2]{MR98k:33040}.  Set
\[
 \operatorname{rk}_\eta(i)=l'_\eta(i)+1.
\]
This is the rank of row $i$ in decreasing order of length, with ties broken
from top to bottom; hence
\begin{equation}
 \overline\eta_i=\alpha\eta_i-\operatorname{rk}_\eta(i)+1.
 \label{eqn-r19-eigenvalue-rank}
\end{equation}
The eigenvalues are pairwise distinct as polynomials in $\alpha$
\label{rem:distinct:eigenvals}.

With the hooks defined above, set
$d_\eta=\prod_{s\in\eta}d_\eta(s)$,
$d'_\eta=\prod_{s\in\eta}d'_\eta(s)$, and $F_\eta=d_\eta E_\eta$.
Each hook polynomial is nonzero, so $d_\eta d'_\eta$ is a unit in
$\mathbb Q(\alpha)$.

For Sahi's scalar product \cite[Theorems~1.1 and~1.2]{MR98g:05154}, the $F_\eta$ are orthogonal and
$\langle F_\eta,F_\eta\rangle=d_\eta d'_\eta$,
hence the dual to $F_\eta$ is
\begin{equation}
 F_\eta^*=(d_\eta d'_\eta)^{-1}F_\eta.
 \label{eqn-r19-dual}
\end{equation}

\subsection{Maximal subsets}

\begin{lemma}\label{lem-r19-maximal-completion}
Let $\eta\in\mathbb N^n$ and
\[
 S=\{i_1<\cdots<i_s\}\ne\emptyset,
 \qquad  \lambda=c_S(\eta).
\]
Adjoin to $S$ all indices in
\begin{equation}
\begin{split}
 &\{i<i_1:\eta_i=\eta_{i_1}\},\\
 &\{i_{t-1}<i<i_t:\eta_i=\eta_{i_t}\}\qquad(2\le t\le s),\\
 &\{i>i_s:\eta_i=\eta_{i_1}+1\}.
\end{split}
\label{eqn-r19-maximal-completion}
\end{equation}
The resulting set $L$ is maximal with respect to $\eta$ and
$c_L(\eta)=\lambda$.

Conversely, if $L$ is maximal, $\lambda=c_L(\eta)$, and
$S=\{i:\eta_i\ne\lambda_i\}$, then $\lambda=c_S(\eta)$ and $L$ is
precisely the completion \eqref{eqn-r19-maximal-completion}.  In
particular, it is the unique maximal subset producing $\lambda$.
If $L=\{j_1<\cdots<j_\ell\}$, then, for every $k$,
\begin{equation}
 j_\ell\ge k
 \quad\Longleftrightarrow\quad
 i_s\ge k\ \hbox{ or }\
 \eta_i=\eta_{i_1}+1\ \hbox{ for some }i\ge k.
\label{eqn-r19-completion-tail}
\end{equation}
\end{lemma}
\begin{proof}
Let $r$ be an index adjoined in
\eqref{eqn-r19-maximal-completion}.  If $r$ is not the last selected index,
the next selected index $r^+$ satisfies $\eta_r=\eta_{r^+}$; if $r$ is
last, then $\eta_r=\eta_{i_1}+1$.  In either case
$(c_L(\eta))_r=\eta_r$, so every adjoined position is fixed.  The values
assigned at the original positions of $S$ are therefore the same for
$c_L(\eta)$ and $c_S(\eta)$, whence
$c_L(\eta)=c_S(\eta)=\lambda$.  Since every index with the relevant
interval value has been inserted, the remaining unselected indices satisfy
(M1)--(M3), so $L$ is maximal.

Conversely, let $L$ be maximal and $\lambda=c_L(\eta)$.  Delete each
selected position with $\lambda_i=\eta_i$.  Deleting such a fixed position
does not alter the cyclic assignment, so the surviving set is exactly
$S=\{i:\eta_i\ne\lambda_i\}$ and still gives $\lambda=c_S(\eta)$.
Every deleted position has the interval value in
\eqref{eqn-r19-maximal-completion}, while maximality forbids an index with
that value from remaining unselected.  Thus the deleted positions are
exactly the three completion sets, proving uniqueness.

If $i_s<k$, completion reaches $k$ or farther precisely when its terminal set
contains some $i\ge k$, equivalently when $\eta_i=\eta_{i_1}+1$ for such an
$i$.  This proves \eqref{eqn-r19-completion-tail}; the case $i_s\ge k$ is
immediate.
\end{proof}

\subsection{Reflection and raising recursions}

For $1\le r<n$, let $s_r$ act on compositions by interchanging entries
$r$ and $r+1$, and on polynomials by interchanging the variables
$x_r$ and $x_{r+1}$.

\begin{lemma}\label{lem-cherednik-row-exchange}
Let $\eta\in\mathbb N^n$, $1\le r<n$, and suppose
$\eta_r\ne\eta_{r+1}$.  Then
$\overline{s_r\eta}=s_r\overline\eta$.
\end{lemma}

\begin{proof}
Put $\gamma=s_r\eta$.  Since $\eta_r\ne\eta_{r+1}$, the exchanged rows cross
no row of the same length.  Hence
\[
 \operatorname{rk}_\gamma(r)=\operatorname{rk}_\eta(r+1),\qquad
 \operatorname{rk}_\gamma(r+1)=\operatorname{rk}_\eta(r),
\]
while $\operatorname{rk}_\gamma(h)=\operatorname{rk}_\eta(h)$ for
$h\ne r,r+1$.  Together with
$\gamma_r=\eta_{r+1}$ and $\gamma_{r+1}=\eta_r$, substitution in
\eqref{eqn-r19-eigenvalue-rank} gives
$\overline\gamma=s_r\overline\eta$.
\end{proof}

\begin{proposition}\label{lem-r19-F-reflection}
Let $\gamma\in\mathbb N^n$, $1\le r<n$, and put
$\delta=\overline\gamma_r-\overline\gamma_{r+1}$.
Then $\delta\ne0$ and
\begin{equation}
 s_rF_\gamma=\frac{1}{\delta}F_\gamma+
 \left(1-\frac{1}{\delta}\right)F_{s_r\gamma}.
 \label{eqn-r19-F-reflection}
\end{equation}
Moreover, the interchange coefficient $1-1/\delta$ is zero if and only if
$\gamma_r=\gamma_{r+1}$.
\end{proposition}

\begin{proof}
Suppose first that $\gamma_r<\gamma_{r+1}$ and set $\nu=s_r\gamma$.
The coefficient of $\alpha$ in $\delta$ is nonzero, so $\delta\ne0$.
By \Cref{lem-cherednik-row-exchange},
\begin{equation}\label{eq:diff:cherednik:sr}
 \overline\nu_r-\overline\nu_{r+1}
 =\overline\gamma_{r+1}-\overline\gamma_r=-\delta.
\end{equation}
Knop--Sahi \cite[Proposition~4.3]{MR98k:33040} gives
$-\delta E_\nu=(-\delta s_r+1)E_\gamma$,
hence
$s_rE_\gamma=\frac{1}{\delta}E_\gamma+E_\nu$.
Multiplying by $d_\gamma$ and applying Sahi's hook-ratio lemma to $\nu$
\cite[Lemma~4.2]{MR98g:05154} gives, by
\cref{eq:diff:cherednik:sr},
\[
 \frac{d_\gamma}{d_\nu}
 =\frac{-\delta+1}{-\delta}
 =\frac{\delta-1}{\delta}.
\]
Therefore
\[
 s_rF_\gamma
 =\frac{1}{\delta}F_\gamma+\frac{d_\gamma}{d_\nu}F_\nu
 =\frac{1}{\delta}F_\gamma+
   \left(1-\frac{1}{\delta}\right)F_\nu,
\]
which is \eqref{eqn-r19-F-reflection}.

If $\gamma_r>\gamma_{r+1}$, set $\nu=s_r\gamma$.  Applying the first case
to $\nu$ gives
\[
 s_rF_\nu=-\frac{1}{\delta}F_\nu+
 \left(1+\frac{1}{\delta}\right)F_\gamma.
\]
Applying $s_r$ and using $s_r^2=1$ gives
\[
 (\delta+1)s_rF_\gamma
 =\delta F_\nu+s_rF_\nu
 =\frac{1}{\delta}(\delta+1)F_\gamma
  +\frac{1}{\delta}(\delta^2-1)F_\nu.
\]
Since the coefficient of $\alpha$ in $\delta$ is
$\gamma_r-\gamma_{r+1}\ne0$, we have $\delta+1\ne0$, and division by
$\delta+1$ gives \eqref{eqn-r19-F-reflection}.

If $\gamma_r=\gamma_{r+1}$, the coleg-length definition gives
$l'_\gamma(r+1)=l'_\gamma(r)+1$, so $\delta=1$.  Also
$s_rF_\gamma=F_\gamma$ by
Knop--Sahi~\cite[Lemma~2.4]{MR98k:33040}, and
\eqref{eqn-r19-F-reflection} follows.

For unequal parts $\delta-1$ has nonzero $\alpha$-coefficient, whereas in the
equal case $\delta=1$.  This also proves the stated vanishing criterion.
\end{proof}

Put
\[
 \Phi(\gamma_1,\ldots,\gamma_n)
  =(\gamma_2,\ldots,\gamma_n,\gamma_1+1),\qquad
 (\Phi f)(x)=x_nf(x_n,x_1,\ldots,x_{n-1}),
\]
and
\[
 \rho(\gamma)=\overline\gamma_1+\alpha+n,\qquad
 \rho'(\gamma)=\overline\gamma_1+\alpha+n-1.
\]

Sahi \cite[Lemma~4.1]{MR98g:05154} proves the upper-hook raising ratio.  The
lower-hook ratio follows from the same box correspondence.

\begin{lemma}
\label{lem-r19-lower-hook-raising}
For every $\gamma\in\mathbb N^n$,
$d'_{\Phi\gamma}/d'_\gamma=\rho'(\gamma)$.
\end{lemma}

\begin{proof}
Use Sahi's box correspondence
\cite[proof of Lemma~4.1]{MR98g:05154}:
\[
 (i,j)\longmapsto
 \begin{cases}
  (i-1,j),&i>1,\\
  (n,j+1),&i=1.
 \end{cases}
\]
It matches the boxes of $\gamma$ with all boxes of $\Phi\gamma$ except
$(n,1)$.  Directly from the definitions, every matched box has the same arm
and leg, hence the same upper and lower hooks.  The new box has arm
$\gamma_1$ and leg $n-1-l'_\gamma(1)$, so its lower hook is
\[
 \alpha(\gamma_1+1)+n-1-l'_\gamma(1)
 =\overline\gamma_1+\alpha+n-1=\rho'(\gamma).
\]
Hence $d'_{\Phi\gamma}=\rho'(\gamma)d'_\gamma$.
\end{proof}

\begin{proposition}
\label{prop-r19-raising}
For every $\gamma\in\mathbb N^n$,
$\rho(\gamma)$ is nonzero and
\begin{equation}
 F_{\Phi\gamma}=\rho(\gamma)\Phi F_\gamma.
 \label{eqn-r19-raising}
\end{equation}
\end{proposition}

\begin{proof}
Knop--Sahi's raising formula
\cite[Corollary~4.2]{MR98k:33040} gives $E_{\Phi\gamma}=\Phi E_\gamma$.
Since $F_\gamma=d_\gamma E_\gamma$, we get
\[
 F_{\Phi\gamma}
 =\frac{d_{\Phi\gamma}}{d_\gamma}\Phi F_\gamma.
\]
Sahi's hook recursion \cite[Lemma~4.1]{MR98g:05154} gives
$d_{\Phi\gamma}/d_\gamma=\rho(\gamma)$,
which proves the identity.  Its $\alpha$-coefficient is $\gamma_1+1>0$,
so $\rho(\gamma)\ne0$.
\end{proof}

\begin{lemma}
\label{lem-r19-Phi-eigenvalues}
For every $\gamma\in\mathbb N^n$,
\[
 \overline{\Phi\gamma}
 =
 \bigl(
  \overline\gamma_2,\ldots,\overline\gamma_n,
  \overline\gamma_1+\alpha
 \bigr).
\]
\end{lemma}

\begin{proof}
Put $\widetilde\gamma=\Phi\gamma$.  For $i>1$, the row of length $\gamma_i$
moves from position $i$ to $i-1$.  The only row whose side relative to it
changes is the old first row, and
\[
 \mathbf1_{\{\gamma_1\ge\gamma_i\}}
 =\mathbf1_{\{\gamma_1+1>\gamma_i\}}.
\]
Hence
\[
 \operatorname{rk}_{\widetilde\gamma}(i-1)
 =\operatorname{rk}_\gamma(i)\qquad(i>1).
\]
For the raised row,
\[
 \operatorname{rk}_{\widetilde\gamma}(n)
 =1+\#\{h>1:\gamma_h\ge\gamma_1+1\}
 =1+\#\{h>1:\gamma_h>\gamma_1\}
 =\operatorname{rk}_\gamma(1).
\]
Substitution in \eqref{eqn-r19-eigenvalue-rank} gives the displayed
transformation; only the raised row has acquired an additional
$\alpha$ in its length term.
\end{proof}


\section{Support of the degree-one product}\label{sec-degree-one-product-supp}

For a polynomial $h$, write $[h]_\mu$ for the coefficient of $F_\mu$ in
$h$.  Its \emph{support} is the set of $\mu$ for which $[h]_\mu\ne0$ in
$\mathbb Q(\alpha)$.  By \eqref{eqn-r19-dual}, the $F$- and $F^*$-supports
coincide.

Put $e_1=x_1+\cdots+x_n$.
\begin{lemma}\label{lem-r19-eone-commutator}
For $1\le i\le n$,
\begin{equation}
 [\xi_i,e_1]=\alpha x_i.
 \label{eqn-r19-eone-commutator}
\end{equation}
Consequently, for $a_1,\ldots,a_n\in\mathbb Q(\alpha)$ and
$\eta,\lambda\in\mathbb N^n$,
\begin{equation}
 \alpha\left[\left(\sum_{i=1}^na_ix_i\right)F_\eta\right]_\lambda
 =\left(\sum_{i=1}^na_i
   (\overline\lambda_i-\overline\eta_i)\right)
   [e_1F_\eta]_\lambda.
 \label{eqn-r19-linear-spectral-coefficient}
\end{equation}
\end{lemma}

\begin{proof}
In the convention of Knop--Sahi \cite[Section~3]{MR98k:33040},
\[
 \xi_i=\alpha x_i\partial_i
  +\sum_{j<i}N_{ij}x_j+\sum_{j>i}x_jN_{ij},
 \qquad
 N_{ij}=\frac{1-s_{ij}}{x_i-x_j}.
\]
Since $e_1$ is symmetric, $N_{ij}(e_1f)=e_1N_{ij}f$; multiplication by
$x_j$ also commutes with $e_1$.  Thus the reflection terms commute with
$e_1$, while
\[
 \alpha x_i\partial_i(e_1f)
 =e_1\alpha x_i\partial_i f+\alpha x_if.
\]
This proves \eqref{eqn-r19-eone-commutator}.  Summing the commutators with
weights $a_i$ and applying the result to $F_\eta$ gives
\[
 \sum_i a_i\xi_i(e_1F_\eta)
 =e_1\sum_i a_i\overline\eta_iF_\eta
   +\alpha\sum_i a_ix_iF_\eta.
\]
Comparison of the $F_\lambda$-coefficients, using
$\xi_iF_\lambda=\overline\lambda_iF_\lambda$, gives
\eqref{eqn-r19-linear-spectral-coefficient}.
\end{proof}

Knop--Sahi's tableau formula \cite[Theorem~5.1]{MR98k:33040}, specialized
to the one-box shape, gives
\begin{equation}
 F_{\varepsilon_k}=(\alpha+k)x_k+\sum_{i>k}x_i,
 \qquad 1\le k\le n.
 \label{eqn-r19-degree-one-explicit}
\end{equation}
In \eqref{eqn-r19-F-reflection}, we call the $F_\gamma$ summand the
\emph{fixed term} and the $F_{s_r\gamma}$ summand the \emph{interchange term}.
Along a formal branch, let
$\sigma=(\sigma_1,\ldots,\sigma_n)$ record source indices by current
position, so initially $\sigma_r=r$.  Write $s_r\sigma$ for the tuple
obtained by swapping $\sigma_r,\sigma_{r+1}$, and
\[
 \Phi\sigma=(\sigma_2,\ldots,\sigma_n,\sigma_1^+),
\]
where the superscript records the unique row that gains one box.  A fixed
choice leaves $\sigma$ unchanged, an interchange at $s_r$ replaces it by
$s_r\sigma$, and raising replaces it by $\Phi\sigma$.

Since $\Phi=x_ns_{n-1}\cdots s_1$, cancellation of adjacent
transpositions gives
\begin{equation}
 x_i=s_i\cdots s_{n-1}\Phi s_1\cdots s_{i-1}.
 \label{eqn-r19-xi-word}
\end{equation}
The operators act on $F_\eta$ in the order
\[
 s_{i-1},s_{i-2},\ldots,s_1,\quad
 \Phi,\quad
 s_{n-1},s_{n-2},\ldots,s_i.
\]
Associate $s_r$ before $\Phi$ with source index $r$, and $s_r$ after
$\Phi$ with source index $r+1$.  These are precisely the source indices
other than $i$.  Expanding by \eqref{eqn-r19-F-reflection}, a subset
$K\ni i$ specifies one choice by taking the fixed term exactly at the
reflections whose source indices lie in $K$; every complete choice arises
uniquely in this way.

\begin{lemma}\label{lem-r19-coordinate-expansion}
For $\eta\in\mathbb N^n$ and $1\le r\le n$,
\begin{equation}
 \{\mu:[x_rF_\eta]_\mu\ne0\}
 =\{c_K(\eta):K\text{ is maximal with respect to }\eta,\ r\in K\},
 \label{eqn-r19-coordinate-support}
\end{equation}
and the compositions on the right are distinct.
\end{lemma}
\begin{proof}
Fix $i$ and let $K=\{j_1<\cdots<j_\ell\}\ni i=j_t$.  Consecutive
interchanges $s_{b-1},\ldots,s_a$ send
\[
 (\sigma_a,\ldots,\sigma_b)\longmapsto
 (\sigma_b,\sigma_a,\ldots,\sigma_{b-1})\qquad(a<b).
\]

The fixed choices therefore cut the reflections before $\Phi$ into the blocks
\[
 [1,j_1],\qquad [j_{p-1}+1,j_p]\quad(2\le p\le t),
\]
processed from $p=t$ down to $p=1$, where $p=1$ denotes $[1,j_1]$.
After $\Phi$, first process $[j_\ell,n]$ and then
\[
 [j_{p-1},j_p-1]\quad(p=\ell,\ell-1,\ldots,t+1).
\]
In each nontrivial block the rightmost source label is selected and the others
are unselected.  The first family puts $j_1$ in position $1$; after $\Phi$
sends it to position $n$ and adds one box, the second family gives
\begin{equation}
 j_{p+1}\longmapsto j_p\quad(p<\ell),\qquad
 j_1^+\longmapsto j_\ell,
 \qquad h\longmapsto h\quad(h\notin K).
 \label{eqn-r19-selected-transport}
\end{equation}
Hence the endpoint is $c_K(\eta)$.

The same rotations show that every unselected row is interchanged exactly
once.  Its selected partner has length
\[
 \eta_{j_1}\quad(h<j_1),\qquad
 \eta_{j_p}\quad(j_{p-1}<h<j_p),\qquad
 \eta_{j_1}+1\quad(h>j_\ell),
\]
respectively.  Fixed factors, reflection denominators, and the raising factor are nonzero; an
interchange factor vanishes exactly when its two row lengths are equal.  Thus the $K$-term is nonzero exactly when
(M1)--(M3) hold, i.e. exactly when $K$ is maximal.

Distinct maximal sets have distinct endpoints by
\Cref{lem-r19-maximal-completion}, so the surviving terms cannot cancel.
This proves \eqref{eqn-r19-coordinate-support}.
\end{proof}

\begin{lemma}\label{lem-r19-spectral-coordinate}
Let $L=\{j_1<\cdots<j_\ell\}$ be maximal with respect to $\eta$ and put
$\lambda=c_L(\eta)$.  Then
\begin{equation}
 \overline\lambda_{j_p}=\overline\eta_{j_{p+1}}\quad(p<\ell),
 \qquad
 \overline\lambda_{j_\ell}=\overline\eta_{j_1}+\alpha,
 \qquad
 \overline\lambda_i=\overline\eta_i\quad(i\notin L).
 \label{eqn-r19-maximal-eigenvalue-transformation}
\end{equation}
Define
\begin{equation}
 q_i=\overline\eta_i-\overline\lambda_i
 =\begin{cases}
   \overline\eta_{j_p}-\overline\eta_{j_{p+1}},&i=j_p,\ p<\ell,\\
   \overline\eta_{j_\ell}-\overline\eta_{j_1}-\alpha,&i=j_\ell,\\
   0,&i\notin L.
  \end{cases}
 \label{eqn-r19-q}
\end{equation}
Then $q_i\ne0$ exactly when $i\in L$, and
$[e_1F_\eta]_\lambda\ne0$.  Moreover,
\begin{equation}
 \frac{[x_iF_\eta]_\lambda}{q_i}
 =\frac{[x_jF_\eta]_\lambda}{q_j}
 =-\frac1\alpha[e_1F_\eta]_\lambda\ne0
 \qquad(i,j\in L).
 \label{eqn-r19-coordinate-common}
\end{equation}
In particular,
\[
 L=\{i:\overline\lambda_i\ne\overline\eta_i\}.
\]
\end{lemma}

\begin{proof}
Fix $i\in L$ and consider the nonzero $L$-term in the expansion of
$x_iF_\eta$.  By maximality every interchange is between unequal rows, so
\Cref{lem-cherednik-row-exchange} applies at each such step: the current
eigenvalue tuple is transformed by the same adjacent transposition as the
source-label tuple $\sigma$.  Fixed steps alter neither tuple, while
\Cref{lem-r19-Phi-eigenvalues} sends
\[
 (\zeta_1,\ldots,\zeta_n)\longmapsto
 (\zeta_2,\ldots,\zeta_n,\zeta_1+\alpha)
\]
at the raising step.  Combining these transformations with
\eqref{eqn-r19-selected-transport} gives
\eqref{eqn-r19-maximal-eigenvalue-transformation}, and then
\eqref{eqn-r19-q}.

For $i\in L$, \Cref{lem-r19-coordinate-expansion} gives
$[x_iF_\eta]_\lambda\ne0$.  The case $a_i=1$ and $a_r=0$ for $r\ne i$ of
\eqref{eqn-r19-linear-spectral-coefficient} is
\begin{equation}
 \alpha[x_iF_\eta]_\lambda
 =-q_i[e_1F_\eta]_\lambda.
 \label{eqn-r19-coordinate-spectral-coefficient}
\end{equation}
The left side is nonzero; hence $q_i\ne0$ for $i\in L$ and
$[e_1F_\eta]_\lambda\ne0$.  Since $q_i=0$ off $L$, division by $q_i$ gives
\eqref{eqn-r19-coordinate-common} and
$L=\{i:\overline\lambda_i\ne\overline\eta_i\}$.
\end{proof}
\begin{proof}[Proof of \Cref{thm-non-van-min}]
By \Cref{lem-r19-coordinate-expansion} and
\eqref{eqn-r19-degree-one-explicit}, a target occurs only as
$\lambda=c_L(\eta)$ with $L$ maximal.  By
\Cref{lem-r19-maximal-completion}, this is equivalent to (a), and $L$ is the
completion of the difference set $S=\{i_1<\cdots<i_s\}$, which is nonempty
because $|\lambda|=|\eta|+1$.

Fix $L=\{j_1<\cdots<j_\ell\}$, write $z_i=\overline\eta_i$, and put
\begin{equation}
 Q_k=(\alpha+k)q_k+\sum_{i>k}q_i.
 \label{eqn-r19-Qk}
\end{equation}
Apply \eqref{eqn-r19-linear-spectral-coefficient} to
\eqref{eqn-r19-degree-one-explicit}.  Since
$q_i=\overline\eta_i-\overline\lambda_i$, it gives
\[
 g_{k\eta}^\lambda
 =-\frac{d_\lambda d'_\lambda}{\alpha}
   [e_1F_\eta]_\lambda Q_k.
\]
By \Cref{lem-r19-spectral-coordinate} the factor preceding $Q_k$ is nonzero.
Since
\begin{equation}
 \sum_{r=p}^{\ell}q_{j_r}=z_{j_p}-z_{j_1}-\alpha,
 \qquad \sum_{r=1}^{\ell}q_{j_r}=-\alpha,
 \label{eqn-r19-q-telescope}
\end{equation}
we have the following.  If $k\notin L$ and $k<j_\ell$, let $j_p$ be the first selected index
larger than $k$.  Then
\begin{equation}
 Q_k=z_{j_p}-z_{j_1}-\alpha.
 \label{eqn-r19-Q-gap}
\end{equation}
(For $p=1$ this is $-\alpha$.)  Moreover
\[
 Q_k=0\quad(k>j_\ell),\qquad
 Q_{j_\ell}=(\alpha+j_\ell)q_{j_\ell},
\]
whereas for $k=j_p<j_\ell$,
\begin{equation}
 Q_k=(\alpha+k)(z_{j_p}-z_{j_{p+1}})
     +z_{j_{p+1}}-z_{j_1}-\alpha.
 \label{eqn-r19-Q-selected}
\end{equation}

If $k\notin L$ and $k<j_\ell$, \eqref{eqn-r19-Q-gap} can vanish only if
$\eta_{j_p}=\eta_{j_1}+1$ and the two row ranks are equal, which is
impossible.  For $k=j_\ell$, nonvanishing follows from
$q_{j_\ell}\ne0$.

Suppose $k=j_p<j_\ell$.  If
$\eta_{j_p}\ne\eta_{j_{p+1}}$, \eqref{eqn-r19-Q-selected} has nonzero
$\alpha^2$-coefficient.  Otherwise maximality gives
$q_{j_p}=1$, and
\[
 Q_k=k+z_{j_{p+1}}-z_{j_1}.
\]
This can vanish only if
$\eta_{j_{p+1}}=\eta_{j_1}=:a$.  In that case the tie-broken row ranks give
\begin{equation}
 Q_k
 =k-\#\{j_1\le h\le k:\eta_h=a\}
 =(j_1-1)+\#\{j_1\le h\le k:\eta_h\ne a\}.
 \label{eqn-r19-Q-count}
\end{equation}
Hence $Q_k=0$ exactly when $j_1=1$ and
$\eta_1=\cdots=\eta_k=a$.  In that event every selected position through
$k$ is fixed, so the first difference position $i_1$ lies after $k$; the
fixed selected chain gives $\eta_{i_1}=a$, and (c) fails.  Conversely, if
(c) fails, \Cref{lem-r19-maximal-completion} inserts $1,\ldots,k$ as fixed
selected positions.  Hence $k=j_p<j_\ell$,
$\eta_{j_p}=\eta_{j_{p+1}}=\eta_{j_1}$, and
\eqref{eqn-r19-Q-count} vanishes.  Thus
\[
 Q_k\ne0\quad\Longleftrightarrow\quad
 j_\ell\ge k\text{ and (c) holds}.
\]
By \eqref{eqn-r19-completion-tail}, $j_\ell\ge k$ is condition~(b), completing
the proof.

\end{proof}

\section{The jeu de fl\`eches Pieri rule}
\label{sec-r19-recursive}

\subsection{Row ranks and hook values}

Let $L=\{j_1<\cdots<j_\ell\}$ be maximal for $\eta$ and put
$\lambda=c_L(\eta)$.  For $p<\ell$, a source arrow hits row $j_p$ precisely
when
$\eta_{j_p}>\eta_{j_{p+1}}$, at column $\eta_{j_{p+1}}+1$; the last
source arrow hits row $j_\ell$ precisely when
$\eta_{j_\ell}>\eta_{j_1}+1$, at column $\eta_{j_1}+2$.  Dually, a target
arrow hits row $j_p$ for $p>1$ precisely when
$\lambda_{j_{p-1}}<\lambda_{j_p}$, at column
$\lambda_{j_{p-1}}+1$, and it hits row $j_1$ precisely when
$\lambda_{j_\ell}-1<\lambda_{j_1}$, at column $\lambda_{j_\ell}$.
If the two relevant lengths agree, there is no hit.

\begin{lemma}\label{lem-r19-rank-displacement}
Let $L$ be maximal with respect to $\eta$, put $\lambda=c_L(\eta)$, and
let $i\in L$.  Then
\begin{equation}
 \operatorname{rk}_\lambda(i)-\operatorname{rk}_\eta(i)
 =\begin{cases}
   l_\eta(i,\lambda_i+1)+1,&\eta_i>\lambda_i,\\
   1,&\eta_i=\lambda_i,\\
   -l_\lambda(i,\eta_i+1),&\eta_i<\lambda_i.
  \end{cases}
 \label{eqn-r19-rank-displacement}
\end{equation}
\end{lemma}

\begin{proof}
Write $i=j_p$, put
$x_r=\eta_{j_r}$, $y_r=\lambda_{j_r}$, and set
$a=x_p=\eta_i$, $b=y_p=\lambda_i$.  Thus
$y_r=x_{r+1}$ for $r<\ell$ and $y_\ell=x_1+1$.  For every integer $t$,
\begin{equation}
 \sum_{r<p}\bigl(\mathbf1_{\{y_r\ge t\}}-
                         \mathbf1_{\{x_r\ge t\}}\bigr)
 +\sum_{r>p}\bigl(\mathbf1_{\{y_r>t\}}-
                         \mathbf1_{\{x_r>t\}}\bigr)
 =\mathbf1_{\{a\ge t\}}-\mathbf1_{\{b>t\}}.
 \label{eqn-r19-rank-cyclic-cancellation}
\end{equation}
The sums telescope, since
$\mathbf1_{\{x_1+1>t\}}=\mathbf1_{\{x_1\ge t\}}$.

Now
\[
 \operatorname{rk}_\lambda(i)-\operatorname{rk}_\eta(i)
 =\sum_{h<i}\bigl(\mathbf1_{\{\lambda_h\ge b\}}-
                         \mathbf1_{\{\eta_h\ge a\}}\bigr)
 +\sum_{h>i}\bigl(\mathbf1_{\{\lambda_h>b\}}-
                         \mathbf1_{\{\eta_h>a\}}\bigr).
\]

At a common threshold the unselected contributions vanish because
$\lambda_h=\eta_h$, while the selected contributions are
\eqref{eqn-r19-rank-cyclic-cancellation}.
If $a>b$, the uncancelled terms are
\[
 1+\#\{h<i:b\le\eta_h\le a-1\}
   +\#\{h>i:b+1\le\eta_h\le a\}
 =l_\eta(i,b+1)+1.
\]
Here \eqref{eqn-r19-rank-cyclic-cancellation} is used with $t=b$;
the remaining terms change the source threshold from $b$ to $a$.
If $a=b$, put $t=a$ in
\eqref{eqn-r19-rank-cyclic-cancellation}; the source and target thresholds
coincide and its right side is $1$, so the rank difference is $1$.  If
$a<b$, use \eqref{eqn-r19-rank-cyclic-cancellation} with $t=a$ and change
the target threshold from $a$ to $b$;
the result is
\[
 -\#\{h<i:a\le\lambda_h\le b-1\}
 -\#\{h>i:a+1\le\lambda_h\le b\}
 =-l_\lambda(i,a+1).
\]
This proves \eqref{eqn-r19-rank-displacement}.
\end{proof}

\begin{lemma}\label{lem-r19-hook-values}
Let $\eta\in\mathbb N^n$, let $L$ be maximal with respect to $\eta$, and
put $\lambda=c_L(\eta)$.  For $i\in L$,
\[
 q_i=
 \begin{cases}
  d_\eta(i,\lambda_i+1),&\eta_i>\lambda_i,\\
  1,&\eta_i=\lambda_i,\\
  -d'_\lambda(i,\eta_i+1),&\eta_i<\lambda_i.
 \end{cases}
\]
Moreover,
\begin{equation}
 b_{\eta\lambda}^{(i)}=q_i b_{\eta\lambda}\qquad(i\in L).
 \label{eqn-r19-shielded-product}
\end{equation}
\end{lemma}

\begin{proof}
By \eqref{eqn-r19-eigenvalue-rank},
\[
 q_i=\alpha(\eta_i-\lambda_i)
     +\operatorname{rk}_\lambda(i)-\operatorname{rk}_\eta(i).
\]
If $\eta_i>\lambda_i$, the arm of $(i,\lambda_i+1)$ in the source is
$\eta_i-\lambda_i-1$; the first case of
\Cref{lem-r19-rank-displacement} therefore gives
$q_i=d_\eta(i,\lambda_i+1)$.  The equal case gives $q_i=1$.  If
$\eta_i<\lambda_i$, the target box $(i,\eta_i+1)$ has arm
$\lambda_i-\eta_i-1$, and the third case gives
$q_i=-d'_\lambda(i,\eta_i+1)$.

These are the entries struck by the arrows.  Shielding row $i$ therefore
changes the ratio of shielded to unshielded products by
\[
 \frac{b_{\eta\lambda}^{(i)}}{b_{\eta\lambda}}
 =
 \begin{cases}
  q_i/1,&\eta_i>\lambda_i,\\
  1,&\eta_i=\lambda_i,\\
  (-q_i)/(-1),&\eta_i<\lambda_i,
 \end{cases}
 =q_i,
\]
which is \eqref{eqn-r19-shielded-product}.
\end{proof}

\subsection{Row hook products}

Put
\[
 U_i(\gamma)=\prod_{j=1}^{\gamma_i}d_\gamma(i,j),\qquad
 V_i(\gamma)=\prod_{j=1}^{\gamma_i}d'_\gamma(i,j).
\]

\begin{lemma}\label{lem-r19-hook-interchange}
Let $\gamma\in\mathbb N^n$ and $1\le r<n$.  If
$\gamma_r\ne\gamma_{r+1}$ and
$\Delta=\overline\gamma_r-\overline\gamma_{r+1}$, then
\[
 \frac{V_r(\gamma)U_{r+1}(\gamma)}
 {U_r(s_r\gamma)V_{r+1}(s_r\gamma)}
 =\frac{\Delta-1}{\Delta}.
\]
\end{lemma}
\begin{proof}
Put $a=\gamma_r$ and $b=\gamma_{r+1}$.  The adjacent rows remain on the
same side of every other row, so hooks outside them do not change.  The
longer row contributes to no leg in the shorter row, before or after the
interchange.  Thus the shorter row is unchanged, while in the longer row
only the box in column $\min(a,b)+1$ changes.

If $a<b$, a direct subtraction of the two adjacent row ranks in
\eqref{eqn-r19-eigenvalue-rank} gives
\[
 d_\gamma(r+1,a+1)=1-\Delta,\qquad
 d_{s_r\gamma}(r,a+1)=-\Delta.
\]
The lower hooks of the shorter row agree, so the stated quotient is
$(1-\Delta)/(-\Delta)$.

If $a>b$, subtracting the adjacent row ranks with threshold $b+1$
instead gives
\[
 d'_\gamma(r,b+1)=\Delta-1,\qquad
 d'_{s_r\gamma}(r+1,b+1)=\Delta.
\]
Now the upper hooks of the shorter row agree, and the quotient is again
$(\Delta-1)/\Delta$.
\end{proof}

\subsection{The coefficient of the filling}

\begin{proposition}
\label{prop-r19-filled-eone}
Let $\eta\in\mathbb N^n$, let $L$ be maximal with respect to $\eta$, put
$\lambda=c_L(\eta)$.  Then
\[
 [e_1F_\eta]_\lambda
 =-\alpha\frac{b_{\eta\lambda}}{d_\lambda d'_\lambda}.
\]
\end{proposition}
\begin{proof}
Write $L=\{j_1<\cdots<j_\ell\}$ and use the unique nonzero $L$-term in
$x_{j_\ell}F_\eta$.  For a state of this term, let $\gamma$ be the current
composition and $\sigma_r$ the source label currently in position $r$, and set
\begin{equation}
 M_L(\gamma,\sigma)
 =
 \prod_{\substack{1\le r\le n\\ \sigma_r\in L}}U_r(\gamma)
 \prod_{\substack{1\le r\le n\\ \sigma_r\notin L}}V_r(\gamma).
 \label{eqn-r19-mixed-row-product}
\end{equation}
Initially $\sigma=\mathrm{id}$.  If $\sigma_{\rm fin}$ is the endpoint tuple,
\eqref{eqn-r19-selected-transport} gives
$\sigma_{{\rm fin},r}\in L\iff r\in L$.  Hence
\[
 M_L(\eta,\mathrm{id})
 =\prod_{j\in L}U_j(\eta)\prod_{j\notin L}V_j(\eta),\qquad
 M_L(\lambda,\sigma_{\rm fin})
 =\prod_{j\in L}U_j(\lambda)\prod_{j\notin L}V_j(\lambda).
\]

For $i=j_\ell$, every interchange in the proof of
\Cref{lem-r19-coordinate-expansion} has an unselected source label on the
left and a selected one on the right.  If it occurs at $s_r$ with current
composition $\gamma$ and
$\Delta=\overline\gamma_r-\overline\gamma_{r+1}$, then
\Cref{lem-r19-hook-interchange} gives
\[
 \frac{M_L(\gamma,\sigma)}
      {M_L(s_r\gamma,s_r\sigma)}
 =
 \frac{V_r(\gamma)U_{r+1}(\gamma)}
      {U_r(s_r\gamma)V_{r+1}(s_r\gamma)}
 =
 \frac{\Delta-1}{\Delta},
\]
which is its interchange coefficient.

Immediately before $\Phi$,
the source label in position $1$ is $j_1\in L$.  The box correspondence in
\Cref{lem-r19-lower-hook-raising} preserves every old upper and lower hook,
and the single new box in that selected row has upper hook $\rho(\gamma)$.
Thus
\[
 \frac{M_L(\gamma,\sigma)}
      {M_L(\Phi\gamma,\Phi\sigma)}
 =\frac1{\rho(\gamma)},
\]
the raising coefficient.

At the fixed reflection indexed by source index $j_p$ $(p<\ell)$, if it
acts at positions $r,r+1$, then
\[
 (\sigma_r,\sigma_{r+1})=(j_p,j_{p+1}).
\]
The preceding unequal interchanges carry the eigenvalues
$\overline\eta_{j_p}$ and $\overline\eta_{j_{p+1}}$ to these positions, so
the fixed coefficient is $q_{j_p}^{-1}$ by \eqref{eqn-r19-q}.  Multiplying
the successive ratios of $M_L$ cancels the intermediate values and gives
\begin{equation}
 [x_{j_\ell}F_\eta]_\lambda
 =
 \frac{M_L(\eta,\mathrm{id})}{M_L(\lambda,\sigma_{\rm fin})}
 \prod_{p<\ell}q_{j_p}^{-1}.
 \label{eqn-r19-canonical-coordinate}
\end{equation}

Before the arrows are launched, the filling product is
\[
 \prod_{j\in L}U_j(\eta)V_j(\lambda)
 \prod_{j\notin L}V_j(\eta)U_j(\lambda).
\]
Dividing by
$d_\lambda d'_\lambda=\prod_jU_j(\lambda)V_j(\lambda)$ gives
\[
 \frac{\displaystyle
   \prod_{j\in L}U_j(\eta)V_j(\lambda)
   \prod_{j\notin L}V_j(\eta)U_j(\lambda)}
  {d_\lambda d'_\lambda}
 =
 \frac{M_L(\eta,\mathrm{id})}{M_L(\lambda,\sigma_{\rm fin})}.
\]
Each selected arrow contributes $q_j^{-1}$ by
\Cref{lem-r19-hook-values}; hence
\begin{equation}
 \frac{b_{\eta\lambda}}{d_\lambda d'_\lambda}
 =
 \frac{M_L(\eta,\mathrm{id})}{M_L(\lambda,\sigma_{\rm fin})}
 \prod_{j\in L}q_j^{-1}.
 \label{eqn-r19-unshielded-hook-telescope}
\end{equation}
Comparison with \eqref{eqn-r19-canonical-coordinate} gives
\[
 [x_{j_\ell}F_\eta]_\lambda
 =q_{j_\ell}\frac{b_{\eta\lambda}}{d_\lambda d'_\lambda}.
\]
Together with
$\alpha[x_{j_\ell}F_\eta]_\lambda
=-q_{j_\ell}[e_1F_\eta]_\lambda$ and $q_{j_\ell}\ne0$, this proves the
proposition.
\end{proof}

\subsection{Multiplication by a variable and the Pieri rule}

\begin{theorem}\label{thm-r19-coordinate-product}
For $\eta\in\mathbb N^n$ and $1\le i\le n$,
\begin{equation}
x_iF_\eta
=\sum_{\substack{L\text{ maximal for }\eta\\ i\in L}}
 b_{\eta,c_L(\eta)}^{(i)}F_{c_L(\eta)}^*.
 \label{eqn-r19-coordinate-product}
\end{equation}
\end{theorem}
\begin{proof}
By \Cref{lem-r19-coordinate-expansion}, the targets are the
$\lambda=c_L(\eta)$ with $L$ maximal and $i\in L$.  For each target,
\Cref{prop-r19-filled-eone} and
\eqref{eqn-r19-coordinate-spectral-coefficient} give
\[
 [x_iF_\eta]_\lambda
 =q_i\frac{b_{\eta\lambda}}{d_\lambda d'_\lambda}
 =\frac{b_{\eta\lambda}^{(i)}}{d_\lambda d'_\lambda},
\]
by \eqref{eqn-r19-shielded-product}.  Since
$F_\lambda=(d_\lambda d'_\lambda)F_\lambda^*$, the dual coefficient is
$b_{\eta\lambda}^{(i)}$, proving \eqref{eqn-r19-coordinate-product}.
\end{proof}

\begin{proof}[Proof of \Cref{thm-r19-pieri}]
Fix a maximal $L$ and put $\lambda=c_L(\eta)$.  With $Q_k$ as in
\eqref{eqn-r19-Qk},
\eqref{eqn-r19-linear-spectral-coefficient} and
\Cref{prop-r19-filled-eone} give
\[
 \alpha[F_{\varepsilon_k}F_\eta]_\lambda
 =-Q_k[e_1F_\eta]_\lambda
 =\alpha Q_k\frac{b_{\eta\lambda}}{d_\lambda d'_\lambda},
 \qquad
 g_{k\eta}^\lambda=Q_kb_{\eta\lambda}.
\]
Using \eqref{eqn-r19-shielded-product} and $q_i=0$ off $L$, the right side is
\[
 \mathbf{1}_{\{k\in L\}}(\alpha+k)b_{\eta,c_L(\eta)}^{(k)}
 +\sum_{\substack{i\in L\\i>k}}b_{\eta,c_L(\eta)}^{(i)}.
\]
The support theorem and \Cref{lem-r19-maximal-completion} account for all
targets, so \eqref{eqn-r19-uniform-coefficient} follows.
\end{proof}

\section{Acknowledgments}
I thank Siddhartha Sahi for helpful discussions and advice, and Dan Marshall
for clarifying an issue in an early version of his paper and pointing me to
the revised version.

\bibliographystyle{abbrv}
\bibliography{waldecks}

\vspace{0.1in}
\noindent\textsc{Universidade Federal de S\~ao Carlos}
\newline Departamento de Matem\'atica
\newline S\~ao Carlos, SP - Brazil.

\end{document}